\documentclass[a4paper, 12pt]{amsart}

\usepackage[utf8]{inputenc}
\usepackage{microtype}
\usepackage{amsmath,amsthm,amssymb}
\usepackage[numbers]{natbib}
\usepackage{hyperref}
\hypersetup{
    colorlinks=true,
    linkcolor=magenta,
    citecolor=blue,
    urlcolor=cyan,
}
\usepackage{graphicx}
\usepackage[margin=20mm]{geometry}

\newtheorem{theorem}{Theorem}[section]
\newtheorem{proposition}[theorem]{Proposition}
\newtheorem{lemma}[theorem]{Lemma}
\newtheorem{corollary}[theorem]{Corollary}
\newtheorem{definition}[theorem]{Definition}
\newtheorem{remark}{Remark}[section]

\newcommand{\E}{\mathbb{E}}
\renewcommand{\P}{\mathbb{P}}
\newcommand{\R}{\mathbb{R}}
\newcommand{\C}{\mathbb{C}}
\newcommand{\Z}{\mathbb{Z}}
\newcommand{\cA}{\mathcal{A}}
\newcommand{\cB}{\mathcal{B}}

\renewcommand{\a}{\mathfrak{a}}
\renewcommand{\b}{\mathfrak{b}}
\newcommand{\e}{\mathfrak{e}}
\renewcommand{\l}{\mathfrak{l}}
\newcommand{\p}{\mathfrak{p}}
\newcommand{\s}{\mathfrak{s}}
\renewcommand{\S}{\mathfrak{S}}
\newcommand{\N}{\mathfrak{N}}
\newcommand{\M}{\mathfrak{M}}
\newcommand{\bone}{\mathbf{1}}
\newcommand{\bzero}{\mathbf{0}}
\newcommand{\rank}{\mathrm{rank}}
\newcommand{\tr}{\mathrm{Tr}}
\newcommand{\bX}{\mathbf{X}}

\newcommand{\bI}{\mathbf{I}}
\newcommand{\bB}{\mathbf{B}}
\newcommand{\bC}{\mathbf{C}}
\newcommand{\bM}{\mathbf{M}}
\newcommand{\bSigma}{\mathbf{\Sigma}}
\newcommand{\MN}{\mathrm{MN}}
\newcommand{\unit}{\bone_{\cA}}
\newcommand{\zero}{\bzero_{\cA}}
\newcommand{\state}{\varphi}
\newcommand{\moeb}{\mathrm{Moeb}}
\newcommand{\leb}{\mathrm{Leb}}
\newcommand{\iid}{\mathrm{i.i.d.}}
\newcommand{\fbern}{Bernoulli}
\newcommand{\fpois}{free Poisson}

\title[On classical and free Poisson thinning]{Some characterization results on classical and free Poisson thinning}
\author[S. S. Mukherjee]{Soumendu Sundar Mukherjee}
\address{
    Interdisciplinary Statistical Research Unit\\
    Indian Statistical Institute\\
    203 B. T. Road, Kolkata 700108
}
\address{
    Department of Mathematics\\
    National University of Singapore\\
    10 Lower Kent Ridge Road, Singapore 119076
}
\email{soumendu041@gmail.com}
\thanks{The author is supported by an INSPIRE Faculty Fellowship from the Department of Science and Technology, Government of India.}

\begin{document}
\maketitle
\begin{abstract}
    Poisson thinning is an elementary result in probability, which is of great importance in the theory of Poisson point processes. In this article, we record a couple of characterization results on Poisson thinning. We also consider several free probability analogues of Poisson thinning, which we collectively dub as \emph{\fpois{} thinning}, and prove characterization results for them, similar to the classical case. One of these \fpois{} thinning procedures arises naturally as a high-dimensional asymptotic analogue of Cochran's theorem from multivariate statistics on the ``Wishart-ness'' of quadratic functions of Gaussian random matrices. We note the implications of our characterization results in the context of Cochran's theorem. We also prove a free probability analogue of Craig's theorem, another well-known result in multivariate statistics on the independence of quadratic functions of Gaussian random matrices.
\end{abstract}

\section{Introduction}\label{sec:intro}
Let $X_1, X_2, \ldots$ be a sequence of i.i.d. Bernoulli$(p)$ random variables, $p \in (0, 1)$. Let $S_n = \sum_{i = 1}^n X_i$ denote the number of successes in $n$ trials. Then, obviously, $S_n$ and $n - S_n$, the number of failures, are not independent. However, if we consider an independent Poisson$(\lambda)$ number $N$ of trials, then $S_N$ and $N - S_N$ become independent. In fact, they are Poisson distributed with parameters $\lambda p$ and $\lambda(1 - p)$, respectively. This result from elementary probability is known as \textit{Poisson thinning} and is of great importance in the theory of Poisson point processes (see, e.g., the book \cite{last2017lectures}).

In Theorem~\ref{thm:conv_classical} below, we note the following (loosely stated) converse of this fact: Let $N$ be a non-negative integer-valued random variable and suppose that $X_1, X_2, \ldots$ are i.i.d. and independent of $N$. For the independence of $S_N$ and $N - S_N$ to hold, (a) $N$ must be Poisson when it is known that $X_1$ is Bernoulli, (b) $X_1$ must be Bernoulli when it is known that $N$ is Poisson. Also, as noted in Proposition~\ref{prop:conv_classical} below, for $S_N$ to be Poisson, (a) $N$ must be Poisson when it is known that $X_1$ is Bernoulli, (b) $X_1$ must be Bernoulli when it is known that $N$ is Poisson.

We then ask if something similar holds in Voiculescu's free probability theory. The answer turns out to be yes. We first consider a free probability analogue of Poisson thinning for which characterization results analogous to the classical case hold, although with a restriction that the rate of the relevant \fpois{} variable must be $1$ (see Proposition~\ref{prop:thinning_free} and Section~\ref{sec:conv_free}). This \fpois{} thinning procedure turns up naturally as a high-dimensional asymptotic version of Cochran's theorem from multivariate statistics on the ``Wishart-ness'' of quadratic functions of Gaussian random matrices. This connection is elaborated in Section~\ref{sec:cochran}. The characterization results we prove on \fpois{} thinning yield characterization results on the asymptotic freeness of such quadratic functions. See Proposition~\ref{prop:asymptotic_cochran_backward_gen}. In course of this, we prove a free probability analogue of Craig's theorem from multivariate statistics on the independence of quadratic functions of Gaussian random matrices.

One can thin a Poisson random variable into more than two independent Poisson random variables using a categorical random variable instead of a Bernoulli. In Section~\ref{sec:more_than_two}, we extend our characterization results to this setting and also discuss the corresponding free probability analogue.

Finally, in Section~\ref{sec:other_versions}, we consider another variant of free Poisson thinning which does not have the restriction that the rate of the relevant \fpois{} variable must be $1$, and prove a characterization result for this variant.
We also mention another possible variant of \fpois{} thinning, suggested by one of the anonymous referees, which too does not have the restriction mentioned above. However, these versions do not seem have interesting random matrix connections as the first one.

We begin by recalling in Section~\ref{sec:prelim} some basic concepts from free probability.

\section{Basic concepts from free probability}\label{sec:prelim}
A \textit{non-commutative probability space} is a pair $(\cA, \state)$, where $\cA$ is an algebra over $\C$ with an identity $\unit$ and $\state : \cA \rightarrow \C$ is a linear functional with $\state(\unit) = 1$. $\state$ is called the \textit{state} and is the analogue of the expectation operator in classical probability. Elements of $\cA$ are the non-commutative analogues of classical random variables. Given such a variable $\a \in \cA$, the numbers $\state(\a^k), k \ge 1$, are called the \textit{moments} of $\a$. The state $\state$ is called \textit{tracial} if $\state(\a\b) = \state(\b\a)$ for all $\a, \b \in \cA$.

Often one imposes further structure on $\cA$. If $\cA$ is a $*$-algebra, i.e. it is equipped with an antilinear operation $* : \cA \rightarrow \cA$, and $\state$ is \textit{positive}, i.e. $\state(\a^* \a) \geq 0$ for all $\a \in \cA$, then $(\cA, \state)$ is called a \textit{$*$-probability space}. In this case, $\state$ is Hermitian, i.e. $\state(\a^*) = \overline{\state(\a)}$ for all $\a \in \cA$. $\state$ is called \textit{faithful} if $\state(\a^*\a) = 0$ only if $\a = \zero$.

If $\cA$ is a $C^*$-algebra, i.e. a $*$-algebra which is also a Banach algebra with respect to some norm $\|\cdot\|$ such that the $C^*$-identity $\|\a^*\a\| = \|\a\|^2$ holds for all $\a \in \cA$, then $(\cA, \state)$ is called a \textit{$C^*$-probability space}.

A natural example comes from the unital $*$-algebra $\mathcal{M}_n(\C)$ of $n\times n$ matrices over $\C$ (the identity matrix serves as the unit while the $*$-operation is given by taking conjugate transpose). This becomes a $*$-probability space when endowed with the state $\state(A) = \frac{1}{n}\tr(A)$ (in fact, a $C^*$-probability space under the operator norm).

To deal with random matrices, given a probability space $(\Omega, \cB, \mathbb{P})$, one may consider the unital $*$-algebra $\mathcal{M}_n(L^{\infty, -}(\Omega, \mathbb{P}))$  of random matrices whose entries are in $L^{\infty, -}(\mathbb{P}) = \cap_{1 \le p < \infty}L^p(\Omega, \mathbb{P})$, the space of random variables with all moments finite. Equipped with the state $\state(A) = \E\frac{1}{n}\tr(A)$, where $\E$ denotes expectation with respect to $\mathbb{P}$, this becomes a $*$-probability space.

We now define the central concept in free probability, namely free independence.

\begin{definition}\label{def:freeness} (Free independence)
	Let $(\cA, \state)$ be a non-commutative probability space and let $\{\cA_i\}_{i \in I}$ be a collection of unital subalgebras of $\cA$, indexed by a fixed set $I$. The subalgebras $\{\cA_i\}_{i\in I}$ are called freely independent, or just free, if 
	\begin{equation*}
	\state(\a_1 \cdots \a_k) = 0
	\end{equation*}
	for every $k \ge 1$, where
	\begin{itemize}
		\item[(a)] $\a_j \in \cA_{i_j}$ for some $i_j \in I$,  
		
		\item[(b)] $\state(\a_j)=0$ for every $1 \le j\le k$, and 
		
		\item[(c)] neighboring elements are from different subalgebras, i.e. $i_1 \neq i_2, i_2 \ne i_3, \ldots, i_{k - 1}\ne i_k$. 
	\end{itemize} 
	Elements $\{\a_i\}_{i \in I}$ from $\cA$ are called freely independent, or just free, if the unital subalgebras generated by each of them are free.
	
	If $(\cA, \state)$ is a $*$-probability space, and $\{\cA_i\}_{i \in I}$ are unital $*$-subalgebras and the above conditions hold, then one says that $\{\cA_i\}_{i \in I}$ are $*$-freely independent, or just $*$-free. Elements $\{\a_i\}_{i \in I}$ from $\cA$ are called $*$-freely independent, or just $*$-free, if the unital $*$-subalgebras generated by each of them are $*$-free.
\end{definition}

Analogous to the classical Kolmogorov construction, which shows the existence of independent random variables, one can show that freely independent variables $\a_1, \ldots, \a_n$ with given moments $\{\state(\a_i^k), 1\leq i \leq n, k \geq 1\}$, or, more generally, freely independent non-commutative probability spaces $(\cA_i, \state_i)$, $i \in I$ exist (see, e.g., Lecture~6 of \cite{nica2006lectures}).

Just as classical independence translates to the vanishing of mixed cumulants, freeness is equivalent to the vanishing of mixed free cumulants, which are free analogues of classical cumulants. We now talk about classical and free cumulants and their relations to classical and free independence.

\subsection{Free cumulants and free independence}\label{sec:freecum}
To define free cumulants, one needs $NC(n)$, the (finite) lattice of all non-crossing partitions of $\{1, \ldots , n\}$ with the reverse refinement partial order. The symbols $\bzero_n$ and $\bone_n$ respectively denote the smallest and the largest partitions $\{\{1\}, \{2\}, \ldots ,\{n\}\}$ and $\{1, 2, \ldots , n\}$. For any $\pi, \sigma\in NC(n)$, we shall denote by $\pi \vee \sigma$ (resp. $\pi \wedge \sigma$) the smallest (resp. largest) element of $NC(n)$ which is larger (resp. smaller) than both $\pi$ and $\sigma$. Being a finite lattice, it is equipped with a unique M\"{o}bius function $\moeb_n$ on $NC^{[2]}(n)=\{(\pi, \sigma) \mid \pi < \sigma, \pi, \sigma\in NC(n)\}$. We denote the natural extension of these functions to $\cup_{n=1}^\infty NC^{[2]}(n)$ by $\moeb$. 

Let $(\mathcal{A},\state)$ be a non-commutative probability space. One can extend $\state$ to a multilinear functional $\state_{n}$ on $\cA^n$:
\begin{equation} \label{eq:mmt4.1}
\state_{n} (\a_1, \a_2, \ldots, \a_n) := \state(\a_1 \a_2 \cdots \a_n).
\end{equation}
The multiplicative extension $\{\state_{\pi} \mid \pi \in NC(n), n \geq 1\}$ of $\{\state_n \mid n \ge 1\}$ is defined as follows:  
If $\pi = \{V_1, V_2, \ldots, V_r \} \in NC(n)$, then
\begin{equation} \label{eq:multi4.1}
\state_{\pi}[\a_1, \a_2, \ldots, \a_n] := \state(V_1)[\a_1, \a_2, \ldots, \a_n] \cdots \state(V_r)[\a_1, \a_2, \ldots, \a_n],
\end{equation}
where 
\[
\state(V)[\a_1, \a_2, \ldots, \a_n] := \state_s(\a_{i_1}, \a_{i_2}, \ldots, \a_{i_s}) = \state(\a_{i_1} \a_{i_2} \cdots \a_{i_s})
\]
for $V = \{i_1, \ldots, i_s\}$ with $i_1 < i_2 < \cdots < i_s$. In particular, 
\begin{equation} 
\state_{\bone_n}[\a_1, \a_2, \ldots, \a_n] = \state_{n}(\a_1, \a_2, \ldots, \a_n) = \state(\a_1 \a_2 \cdots \a_n).
\end{equation}

We may now define free cumulants by using the M\"{o}bius function $\moeb$. 
\begin{definition} (Free cumulants)
	Let $\a_1, \ldots, \a_n \in \cA$. Their \textit{joint free cumulant} \index{free cumulant} of order $n$ is the multilinear functional
	\begin{equation} \label{eq:freecum}
	\kappa_{n}(\a_1, \a_2, \ldots, \a_n) = \sum_{\pi \in NC(n)} \moeb(\pi, \bone_{n})\state_{\pi}[\a_1, \a_2, \ldots, \a_n]. 
	\end{equation}
\end{definition}
In particular, $\kappa_{n}(a) := \kappa_n(a, a, \ldots, a)$ is called the $n$-th free cumulant of $a$. 

As in \eqref{eq:multi4.1}, the functionals $\{\kappa_{n} \mid n \ge 1\}$ also have a multiplicative extension $\{\kappa_{\pi} \mid \pi \in NC(n), n \ge 1\}$. By M\"{o}bius inversion, one has  (see Proposition $11.4$ in \cite{nica2006lectures}), for all $\a_1, \ldots, \a_n$, 
\begin{align*}
\kappa_{\sigma}[\a_{1}, \a_{2}, \ldots, \a_{n}] &= \sum_{\substack{\pi \in NC(n) \\ \pi \le \sigma}} \moeb(\pi, \sigma) \state_{\pi}[ \a_{1}, \a_{2}, \ldots, \a_{n}] \\
\state_{\sigma}[\a_1, \a_2, \ldots, \a_n] &= \sum_{\substack{\pi \in NC(n) \\ \pi \le \sigma}} \kappa_{\pi}[\a_1, \a_2, \ldots, \a_n].
\end{align*}
In particular,
\[
\state(\a_1 \a_2 \cdots \a_n) = \sum_{\pi \in NC(n)} \kappa_{\pi}[\a_1, \a_2, \ldots, \a_n].
\]
Freeness is equivalent to the vanishing of mixed free cumulants.
\begin{proposition}\label{prop:freeness_equivalence}
	Suppose $(\cA, \state)$ is a non-commutative probability space. Then subalgebras $\{\cA_{i}\}_{i \in I}$ of $\cA$ are free if and only if, for all $\a_1, \a_2, \ldots, \a_n \in \cup_{i \in I} \cA_i$, $n \ge 2$, 
	\[
	\kappa_n(\a_1, \a_2, \ldots, \a_n) = 0
	\]
	whenever at least two of the $\a_i$'s are from different $\cA_i$'s.
\end{proposition}
We will need to be able to compute the free cumulants of products of free variables. In this regard, the following formula will be crucial (see Theorem 14.14 of \cite{nica2006lectures}): Suppose $\{\a_1, \ldots, \a_n\}$ and $\{\b_1, \ldots, \b_n\}$ are freely independent. Then, for all $n \ge 1$,
\begin{equation}\label{eq:cumulant_two}
\kappa_n(\a_1\b_1, \ldots, \a_n\b_n) = \sum_{\pi \in NC(n)} \kappa_{\pi}[\a_1, \ldots, \a_n] \kappa_{K(\pi)}[\b_1, \ldots, \b_n].
\end{equation}
Here $K$ is a bijection on $NC(n)$, known as the \textit{Kreweras complementation map}. Consider the following totally ordered set of $2n$ elements: $\{1, \bar{1}, 2, \bar{2}, \ldots, n, \bar{n} \}$. Let $\pi$ be a non-crossing partition of $\{1, \ldots, n\}$. Then $K(\pi)$ is defined as the largest non-crossing partition $\sigma$ of $\{\bar{1}, \ldots, \bar{n}\}$ such that $\pi \vee \sigma \in NC(2n)$. In particular, $K(\bzero_n) = \bone_n$ and $K(\bone_n) = \bzero_n$. For any $\pi$ in $NC(n)$, one has $|K(\pi)| + |\pi| = n + 1$.

\subsection{Joint distribution and convergence}
Let $(\cA, \state)$ be a non-commutative probability space. The joint distribution of $\a_1, \ldots, \a_k \in \cA$ is the linear functional $\mu : \C\langle X_1, \ldots, X_k\rangle \rightarrow \C$ given by
\[
    \mu(P) = \state(P(\a_1, \ldots, \a_k)),
\]
where $\C\langle X_1, \ldots, X_k \rangle$ is the set of all non-commutative polynomials in the indeterminates $X_1, \ldots, X_k$.

Let $(\cA_n, \state_n), n \ge 1$, be non-commutative probability spaces. Let $\a_1^{(n)}, \ldots, \a_k^{(n)} \in \cA_n$, with joint distribution $\mu_n$. We say that $\a_1^{(n)}, \ldots, \a_k^{(n)}$ converges in distribution in the algebraic sense (or just converges for our purposes) to $\a_1, \ldots, \a_k$ if 
\[
    \mu_n(P) \rightarrow \mu(P) \text{ for all } P \in \C\langle X_1, \ldots, X_k \rangle.
\]
If the limiting variables $\a_1, \ldots, \a_k \in \cA$ are freely independent, then we say that $\a_1^{(n)}, \ldots, \a_k^{(n)}$ are \textit{asymptotically free}.
In case of $*$-probability spaces, we have the same definition, except now $\C\langle X_1, \ldots, X_k \rangle$ is replaced by $\C\langle X_1, X_1^*, \ldots, X_k, X_k^*\rangle$. 
\section{Free Poisson thinning}\label{sec:free_thinning}
Let $(\cA, \state)$ be a non-commutative probability space. We say that $\s \in \cA$ is a \textit{semi-circular}\footnote{It is customary to define such variables as self-adjoint elements of a $*$-probability space. However, since our main arguments will not involve the $*$-operation, and will only make use of purely algebraic relations between moments and free cumulants, we have chosen to work in vanilla non-commutative probability spaces. We will mention what happens when we make extra assumptions such as self-adjointness of variables, faithfulness of the state, etc.} variable with mean $0$ and variance $\alpha^2$ if, for all $k \ge 1$,
\[
\state(\s^{2k}) = \alpha^{2k} \#NC(k) = \frac{\alpha^{2k}}{k + 1}\binom{2k}{k},
\]
and $\state(\s^{2k - 1}) = 0$. The moments of a semi-circular variable determine a unique compactly-supported probability measure called the semi-circular law that arises in random matrix theory as the limiting spectral distribution of Wigner matrices.

We say that $\p \in \cA$ is a \textit{\fpois{}} variable with \textit{rate} $\lambda > 0$ and \textit{jump size} $\alpha > 0$ if its free cumulants are given by
\[
\kappa_n(\p) = \kappa_n(\p, \ldots, \p) = \lambda \alpha^n, n \ge 1.
\]
Let $\s$ be a semi-circular variable with mean $0$ and variance $\alpha^2$. Then it is known that, for any idempotent $\a \in \cA$ (i.e. $\a^2 = \a$) with $\state(\a) = \lambda > 0$, the element $\s\a\s$ is a \fpois{} variable with rate $\lambda$ and jump size $\alpha$. In particular, $\s^2$ is a \fpois{} variable with rate $1$. If $\alpha = 1$, then we call $\p$ a standard \fpois{}. The moments of a \fpois{} variable determine a unique compactly-supported probability measure related to the famous \textit{Mar\u{c}enko-Pastur} law in random matrix theory, arising as the limiting spectral distribution of high-dimensional sample covariance matrices. If $\bX$ is an $n \times m$ matrix with standard Gaussian entries, and $\frac{m}{n} \rightarrow c \in (0, \infty)$, then the limiting spectral distribution of the sample covariance matrix $\frac{1}{n} \bX^\top \bX$ is the Mar\u{c}enko-Pastur law with parameter $c$. This is also the law determined by the moments of a \fpois{} variable with rate $\frac{1}{c}$ and jump size $c$. 

We say that $\b \in \cA$ is a \textit{\fbern{}} variable with parameter $p \in [0, 1]$ if all of its moments $\state(\b^n)$, $n \ge 1$, are equal to $p$. If $\b$ is a \fbern{} variable with parameter $p$, then it can be easily seen that $\unit - \b$ is a \fbern{} variable with parameter $1 - p$. Clearly, any idempotent element $\b$ is \fbern{} if $\state(\b) \in [0, 1]$.

\begin{proposition}
	An idempotent $\b$ in a $*$-probability space $(\cA, \state)$ is a \fbern{} if any one of the following two conditions holds.
	\begin{enumerate}
		\item[(a)] $\b$ is a projection, i.e. a self-adjoint idempotent element. 
		\item[(b)] $(\cA, \state)$ is a tracial $C^*$-probability space.
	\end{enumerate}
\end{proposition}
\begin{proof}
	(a) If $\b$ is a projection, then $\state(\b) = \state(\b^2) = \state(\b^*\b) \ge 0$. Also, in that case, $\unit - \b$ is a projection so that $1 - \state(\b) = \state(\unit - \b) \ge 0$. Therefore $\state(\b) \in [0, 1]$.
	
	(b) In a $C^*$-algebra, every idempotent is similar to a self-adjoint one, i.e. there exists an invertible element $\e$ and a projection $\tilde{\b}$ such that $\b = \e \tilde{\b} \e^{-1}$. In fact, one can take $\tilde{\b} = \b\b^* (\unit + (\b - \b^*)(\b^* - \b))^{-1}$. (See, e.g., Lemma~16 of \cite{kaplansky1953modules}.) If $\state$ is tracial, then $\state(\b) = \state(\e \tilde{\b} \e^{-1}) = \state(\e^{-1} \e \tilde{\b}) = \state(\tilde{\b}) \ge 0$.
\end{proof}

Let $\p$ be a \fpois{} variable, and $\b$ a \fbern{} variable. Then $\p\b$ (or $\b\p$) may be thought of as a free analogue of a thinned Poisson variable in classical probability\footnote{A (loose) analogy goes like this: $S_N$ is the sum of $N$ independent copies of $X_1$. Now if things were not random, we could write $S_N$ as $NX_1$. Although this is not actually true, we at least have the Wald identity $\E S_N = \E N \E X_1$. There is another good reason to call this procedure \fpois{} \emph{thinning}, which comes from random matrix theory considerations. The rate parameter $1$ of the \fpois{} variable $\p$ may be thought of as the asymptotic ``scaled rank'' of a suitable random quadratic form. The Bernoulli variable $\b$ may be thought of as a projection which reduces/thins this scaled rank to $p$. See Proposition~\ref{prop:asymptotic_cochran_forward} and Remark~\ref{rem:thinning_justification}.}.

\begin{proposition}[Free Poisson thinning]\label{prop:thinning_free}
	Let $\p$ be a \fpois{} variable with rate $\lambda > 0$ and jump size $\alpha > 0$. Let $\b$ be a \fbern{} with parameter $p \in (0, 1)$. 
	\begin{enumerate}
		\item[(a)] If $\lambda = 1$, then $\p\b$ and $\p(\unit - \b)$ are free. $\p\b$ is a \fpois{} with rate $p$ and $\p(\unit - \b)$ is a \fpois{} with rate $(1 - p)$.
		\item[(b)] If $\lambda \ne 1$, then neither are $\p\b$ and $\p(\unit - \b)$ \fpois{} nor are they freely independent.
	\end{enumerate}
\end{proposition}

\begin{proof}
	(a) That $\p\b$ is \fpois{} with rate $p$ is well-known (see, e.g., Proposition~12.18 of \cite{nica2006lectures}). We nevertheless give the details because part of the calculation will be used for showing (b). We may assume that $\alpha = 1$. Using \eqref{eq:cumulant_two}, we compute
	\begin{align*} 
	\kappa_n(\p\b) &= \sum_{\pi \in NC(n)} \kappa_{\pi}[\p, \ldots, \p] \kappa_{K(\pi)} [\b, \ldots, \b] \\
	&= \sum_{\pi \in NC(n)} \lambda^{|\pi|} \kappa_{K(\pi)}[\b, \ldots, \b] \\
	&= \sum_{\pi \in NC(n)} \lambda^{n + 1 - |K(\pi)|} \kappa_{K(\pi)}[\b, \ldots, \b] \\
	&= \sum_{\pi \in NC(n)} \lambda^{n + 1 - |\pi|} \kappa_{\pi}[\b, \ldots, \b].
	\end{align*}
	If $\lambda = 1$, then, for all $n \ge 1$, $\kappa_n(\p\b) = \sum_{\pi \in NC(n)} \kappa_{\pi}[\b, \ldots, \b] = \state(\b^n) = p$.
    Hence $\p\b$ is \fpois{} with rate $p$. Similarly, $\p(\unit - \b)$ is \fpois{} with rate $(1 - p)$. Finally, to show the freeness of $\p\b$ and $\p(\unit - \b)$, we will use \eqref{eq:cumulant_two} to compute their mixed free cumulants and show that they vanish. For $m \ge 2$, let $\b_{i_1}, \ldots, \b_{i_m} \in \{\b, \unit - \b\}$. We have
    \begin{align*}
        \kappa_m(\p\b_{i_1}, \ldots, \p\b_{i_m}) &= \sum_{\pi \in NC(m)} \kappa_{\pi}[\p, \ldots, \p] \kappa_{K(\pi)}[\b_{i_1}, \ldots, \b_{i_m}] \\
                                                 &= \sum_{\pi \in NC(m)} \kappa_{K(\pi)}[\b_{i_1}, \ldots, \b_{i_m}] \\
                                                 &= \state_m(\b_{i_1}, \ldots, \b_{i_m}) = \state(\b_{i_1}\cdots\b_{i_m}).
    \end{align*}
    If $\kappa_m(\p\b_{i_1}, \ldots, \p\b_{i_m})$ is a mixed free cumulant, then, for some $m_1, m_2 \ge 1, m_1 + m_2 = m$, one can write
    \begin{align*}
        \state(\b_{i_1}\cdots\b_{i_m}) &= \state(\b^{m_1}(\unit - \b)^{m_2}) \\
                                       &= \sum_{j = 0}^{m_2} \binom{m_2}{j}(-1)^j \state(\b^{m_1 + j}) \\
                                       &= p \sum_{j = 0}^{m_2} \binom{m_2}{j}(-1)^j = 0.
    \end{align*}
    This completes the proof of (a).
	
	(b) If $\lambda \ne 1$, let us explicitly compute the first few free cumulants of $\p\b$. Clearly, $\kappa_1(\p\b) = \state(\p\b) = \state(\p) \state(\b) = \lambda p$. Now
	\begin{align*}
	\kappa_2(\p\b) &= \lambda^2 \kappa_2(\b) + \lambda (\kappa_1(\b))^2 \\
	&= \lambda^2 (p - p^2) + \lambda p^2 \\
	&= \lambda p (\lambda (1 - p) + p).
	\end{align*}
	Similarly, 
	\[
	\kappa_2(\p(\unit - \b)) = \lambda (1 - p) (\lambda p + (1 - p)).
	\]
	Thus
	\begin{align*}
	\kappa_2(\p\b) + \kappa_2(\p(\unit - \b)) &= \lambda + 2\lambda(\lambda - 1) p (1 - p) \\
	&\ne \lambda = \kappa_2(\p).
	\end{align*}
	Thus $\p\b$ and $\p(\unit - \b)$ are not free.
	
	If $\p\b$ were \fpois{} with rate $\mu > 0$ and jump size $\beta > 0$, then we would have $\kappa_n(\p\b) = \mu \beta^n$ for all $n \ge 1$. So we must have
	\begin{align*}
	\mu\beta &= \lambda p, \\
	\nu \beta^2 &= \lambda p (\lambda(1 - p) + p).
	\end{align*}
	This means $\nu = \frac{\lambda p}{\lambda(1 - p) + p)}$, $\beta = \lambda(1 - p) + p \ne 1$. Now
	\begin{align*}
	\kappa_3(\p\b) &= \lambda^3 \kappa_3(\b) + 3 \lambda^2 \kappa_2(\b) \kappa_1(\b) + \lambda (\kappa_1(\b))^3 \\
	&= \lambda^3 (p - 3 p^2 + 2 p^3) + 3 \lambda^2 p(1 - p) p + \lambda p^3 \\
	&= \lambda p (\lambda(1 - p) + p)^2 - \lambda^2 p^2 (\lambda(1 - p) + p) + \lambda^2 p^2 \\
	&= \nu \beta^3 - \nu^2 \beta^3 + \nu^2\beta^2 \\
	&= \nu \beta^3 + \nu^2 \beta^2 (1 - \beta) \ne \nu \beta^3.
	\end{align*}
	where in the second equality we have used the moment-free cumulant relation $\kappa_3(\b) = \state(\b^3) -3 \state(\b) \state(\b^2) + 2 (\state(\b))^3$. This shows that $\p\b$ is not \fpois{}, and, similarly, neither is $\p(\unit - \b)$.
\end{proof}
\begin{remark}
	In Proposition~\ref{prop:thinning_free}, instead of just \fbern{} variables, we can have any $\b$ such that $\state(\b^n) = \state(\b) \in \C \setminus \{0, 1\}$ for all $n \ge 1$. In particular, all idempotents are allowed. Then $\p\b$ and $\p(\unit - \b)$ will be freely independent or not, depending on whether $\lambda = 1$ or not. In the former case, $\kappa_n(\p\b) = \state(\b)$ and $\kappa_n(\p(\unit  - \b)) = 1 - \state(\b)$ for all $n \ge 1$.
\end{remark}
\begin{remark}
	Suppose $\b$ is \fbern{} with parameter $p \in \{0, 1\}$. Because of symmetry, consider the case $p = 1$. Due to the freeness of $\p$ and $\b$, all of their joint moments and free cumulants are calculable in terms of their individual moments and free cumulants. Since the moments of $\b$ match those of $\unit$, we can replace $\b$ with $\unit$ in our calculations. Then, for any $\lambda > 0$, $\p\b$ and $\p(\unit - \b)$ are trivially freely independent, $\p\b$ is \fpois{} with parameter $\lambda$, and all free cumulants and hence moments of $\p(\unit - \b)$ vanish so that it is \fbern{} with parameter $0$.
\end{remark}
\begin{remark}
	Instead of $\p\b$ and $\p(\unit - \b)$, we can also phrase everything in terms of $\b\p$ and $\b(\unit - \p)$.
\end{remark}

Thus, unlike the classical case, we need the rate parameter of our \fpois{} variable to be $1$ for thinning to happen\footnote{One of the anonymous referees pointed out that by mixing classical and free independence, one can have a thinning procedure that works for any $\lambda \in [0, 1]$. Let $\p$ be a \fpois{} with parameter $\lambda \in [0, 1]$. Then, as discussed earlier, we can represent $\p$ as $\s^2 \l$, where $\s$ is a standard semi-circular variable and $\l$ is a \fbern{} variable with parameter $\lambda \in [0, 1]$. Now let $\b$ be a \fbern{} variable with parameter $p \in [0, 1]$ which is freely independent from $\s$ but classically independent from $\l$. Then, clearly, $\l\b$ is \fbern{} with parameter $\lambda p$ and is freely independent from $\s$. Thus $\p\b = \s^2\l\b$ is \fpois{} with parameter $\lambda p$. Similarly, $\p(\unit - \b)$ is \fpois{} with parameter $\lambda(1 - p)$. Moreover, by a calculation similar to the one in the proof of Proposition~\ref{prop:thinning_free}, one can easily verify that $\p\b$ and $\p(\unit - \b)$ are freely independent. 

In Section~\ref{sec:other_versions}, we will consider another variant of \fpois{} thinning which also does not have the restriction $\lambda = 1$.}. In Theorem~\ref{thm:conv_free} below, we prove a converse of this result: If $\p$ and $\b$ are two free elements in a non-commutative probability space $(\cA, \state)$ such that $\p\b$ and $\p(\unit - \b)$ are free, then (a) $\p$ must be \fpois{} with rate $1$ and jump size $\state(\p)$ when it is known that $\b$ is \fbern{}, and (b) $\b$ must be a \fbern{} when it is known that $\p$ is \fpois{} with rate $1$.

\subsection{Connections with Cochran's theorem}\label{sec:cochran}
The free analogue of Poisson thinning arises as an asymptotic version of a famous result in multivariate statistics known as Cochran's theorem. Let $\bX$ be an $n \times m$ random matrix whose rows are i.i.d. $m$-variate Gaussians with mean zero and covariance matrix $\bSigma$. We will express this as $\bX \sim \MN_{n, m}(\bzero, \bSigma)$. Then the distribution of the random matrix $\bX^\top \bX$ is called the $m$-dimensional Wishart distribution with $n$ degrees of freedom and covariance matrix $\bSigma$, denoted as $\mathrm{Wishart}_m(\bSigma, n)$. The following result is known as Cochran's theorem\footnote{Cochran only considered the one-dimensional version in his paper \cite{cochran1934distribution}.} (see, e.g., Theorem 3.4.4 of \cite{mardia1979multivariate}).

\begin{proposition}\label{prop:cochran}
	Let $\bX \sim \MN_{n, m}(\bzero, \bSigma)$ and $\bB$ be an $n \times n$ symmetric matrix. Then $\bX^\top \bB \bX$ is Wishart distributed if and only if $\bB$ is idempotent, in which case $\bX^\top \bB \bX \sim \mathrm{Wishart}_m(\bSigma, \rank(\bB))$. 
\end{proposition}

If $\bB$ is idempotent, then $\bB(\bI - \bB) = 0$ so that $\bB \bX$ and $(\bI - \bB) \bX$ are independent Gaussian matrices. This gives the following result. 
\begin{corollary}\label{cor:cochran}
	Let $\bX \sim \MN_{n, m}(\bzero, \bSigma)$ and $\bB$ be an $n \times n$ symmetric idempotent matrix. Then $\bX^\top \bB \bX$ and $\bX^\top (\bI - \bB) \bX$ are independent, with $\bX^\top \bB \bX$ and $\bX^\top(\bI - \bB) \bX$ having respectively $\mathrm{Wishart}_m(\bSigma, \rank(\bB))$ and $\mathrm{Wishart}_m(\bSigma, n - \rank(\bB))$ distributions.
\end{corollary}

In fact, the following result known as Craig's theorem\footnote{The proof given in Craig's paper \cite{craig1943note} is incomplete; a complete proof was given in \cite{ogawa1949independence}.} (see, e.g., Theorem~3.4.5 of \cite{mardia1979multivariate}) gives necessary and sufficient conditions for the independence of such quadratic functions. 
\begin{proposition}[Craig's theorem]
	Let $\bX \sim \MN_{n, m}(\bzero, \bSigma)$ and $\bC_1, \ldots, \bC_k$ be symmetric matrices. Then the random matrices $\bX^\top \bC_1 \bX, \ldots, \bX^\top \bC_k \bX$ are independent if and only if $\bC_r\bC_s = \bzero$ for all $r \ne s$. 
\end{proposition}
 Using this, we have the following result.
\begin{proposition}\label{prop:craig}
	Let $\bX \sim \MN_{n, m}(\bzero, \bSigma)$ and $\bB$ be an $n \times n$ symmetric matrix. Then $\bX^\top \bB \bX$ and $\bX^\top (\bI - \bB) \bX$ are independent if and only if $\bB$ is idempotent.
\end{proposition}
Consider now the high-dimensional asymptotic regime $\frac{m}{n} \rightarrow c \in (0, \infty)$. Assume that $\bX_n \sim \MN_{n, m}(\bzero, \bI)$ and $(\bB_n)_{n \ge 1}$ is a sequence of symmetric idempotent matrices such that $\frac{\rank(\bB_n)}{n} \rightarrow p \in (0, 1)$. Then, it is a well-known result that, as members of the non-commutative probability space $(\mathcal{M}_n(L^{\infty, -}(\Omega, \mathbb{P})), \E \frac{1}{n}\tr)$, the sample covariance matrix $\frac{1}{m}\bX_n\bX_n^\top$ and the deterministic matrix $\bB_n$ converge jointly to a pair $(\p, \b)$ in some non-commutative probability space $(\cA, \state)$, where $\p$ is \fpois{} with rate $c$ and jump size $\frac{1}{c}$, $\b$ is \fbern{} with parameter $p$, and $\p, \b$ are freely independent\footnote{As the law of $\frac{1}{m}\bX_n\bX_n^\top$ is invariant under orthogonal conjugation, asymptotic freeness of $\frac{1}{m}\bX_n \bX_n^\top$ and $\bB_n$ is a simple consequence of Theorem~5.2 of \cite{collins2006integration}, for instance, via Proposition~2.9 of \cite{collins2021asymptotic}.}. That is,
\[
    \E\frac{1}{n}\tr (P(\frac{1}{m}\bX_n\bX_n^\top, \bB_n)) \rightarrow \state(P(\p, \b))
\]
for any polynomial $P$ in two-variables. Using this we can compute the joint distribution of $\frac{1}{n}\bX_n^\top\bB_n\bX_n$ and $\frac{1}{n}\bX_n^\top(\bI_n - \bB_n) \bX_n$. For any polynomial $P$ in two variables, we have
\begin{align*}
\E \frac{1}{m}\tr (P(\frac{1}{n}\bX_n^\top \bB_n \bX_n, &\frac{1}{n}\bX_n^\top (\bI_n - \bB_n) \bX_n)) \\
&= \E \frac{1}{m} \tr (P(\frac{1}{n}\bX_n\bX_n^\top \bB_n, \frac{1}{n}\bX_n\bX_n^\top (\bI_n - \bB_n))) \\
&= \frac{n}{m} \E \frac{1}{n}\tr (P(\frac{m}{n}\frac{1}{m}\bX_n\bX_n^\top \bB_n, \frac{m}{n}\frac{1}{m}\bX_n\bX_n^\top (\bI_n - \bB_n))) \\
&\rightarrow \frac{1}{c} \state (P(c\p\b, c\p(\unit - \b))).
\end{align*}
It follows that when $c = 1$, the two independent matrices $(\frac{1}{n}\bX_n^\top \bB_n \bX_n, \frac{1}{n}\bX_n^\top (\bI_n - \bB_n) \bX_n$) converge jointly to $(\p\b, \p(\unit - \b))$ which are freely independent standard \fpois{} variables with rates $p$ and $(1 - p)$, respectively. We summarize this in the following proposition.

\begin{proposition}[Asymptotic version of Corollary~\ref{cor:cochran}]\label{prop:asymptotic_cochran_forward}
	Suppose $m/n \rightarrow 1$, $\bX_n \sim \MN_{n, m}(\bzero, \bI)$, and $\bB_n$ is symmetric and idempotent with $\frac{\rank(\bB_n)}{n} \rightarrow p \in (0, 1)$. Then the two matrices $(\frac{1}{n}\bX_n^\top \bB_n \bX_n, \frac{1}{n}\bX_n^\top (\bI_n - \bB_n) \bX_n$) converge jointly to two freely independent standard \fpois{} variables with rates $p$ and $(1 - p)$, respectively.
\end{proposition}

\begin{remark}\label{rem:thinning_justification}
    Proposition~\ref{prop:asymptotic_cochran_forward} gives another justification for the name \fpois{} thinning. The projection matrix $\bB_n$ essentially thins/reduces the degrees-of-freedom/rank of the quadratic form $\bX_n^\top \bX_n$. The rate parameters $1$, $p$ and $1 - p$ of the \fpois{} variables $\p$, $\p\b$ and $\p(\unit - \b)$ are the (asymptotic) scaled ranks of the quadratic forms $\frac{1}{n}\bX_n^\top \bX_n$, $\frac{1}{n}\bX_n^\top \bB_n \bX_n$ and $\frac{1}{n}\bX_n^\top (\bI_n - \bB_n) \bX_n$, respectively. The \fbern{} variable $\b$ may thus be thought of as thinning/reducing the ``scaled rank'' (i.e. rate) of the \fpois{} variable $\p$.   
\end{remark}

In fact, we do not need $\bB_n$ to be idempotent, all we need is that $\E \frac{1}{n}\tr\bB_n^k \rightarrow p \in (0, 1)$ for all $k \ge 1$. We also do not need $\bX_n$ to be a Gaussian random matrix---asymptotic freeness of $\frac{1}{m}\bX_n\bX_n^\top$ and $\bB_n$ holds under much more general assumptions on the entries of $\bX_n$, provided one has more control on the entries of $\bB_n$. For example, consider the following two assumptions\footnote{See the proof of Theorem 5.4.5 of \cite{anderson2010introduction} for an illustration of how these assumptions help in showing the asymptotic freeness of Wigner and deterministic matrices. Similar arguments can be used to show the asymptotic freeness of $\frac{1}{m}\bX_n \bX_n^\top$ and $\bB_n$ under these assumptions.}:
\begin{enumerate}
	\item[(A1)] The entries of $\bX_n$ are independent zero mean unit variance random variables with
	\[
	\sup_{n \ge 1} \sup_{\substack{1 \le i \le n \\ 1 \le j \le m}} \E |\bX_n(i, j)|^k < \infty
	\]
	for all $k \ge 1$.
	\vskip5pt
	\item[(A2)] The matrices $\bB_n$, $n \ge 1$, are such that
	\[
	\sup_{k \ge 1} \sup_{n \ge 1} \frac{1}{n} \max \{ (\tr(|\bB_n|^k))^{\frac{1}{k}},  (\tr(|\bI - \bB_n|^k))^{\frac{1}{k}}\} < \infty,
	\]
	where $|\bM| := (( |\bM(i, j)| ))$.
\end{enumerate}
Then the following generalization of Proposition~\ref{prop:asymptotic_cochran_forward} holds.

\begin{proposition}\label{prop:asymptotic_cochran_forward_gen}
	Suppose $(\bX_n)_{n \ge 1}$ is a sequence of $n \times m$ random matrices whose entries satisfy Assumption (A1), $\frac{m}{n} \rightarrow 1$, and $(\bB_n)_{n \ge 1}$ is a sequence of symmetric deterministic matrices satisfying Assumption (A2) such that $\E \frac{1}{n} \tr\bB_n^k \rightarrow p \in (0, 1)$ for all $k \ge 1$. Then the two matrices $(\frac{1}{n}\bX_n^\top \bB_n \bX_n, \frac{1}{n}\bX_n^\top (\bI_n - \bB_n) \bX_n$) converge jointly to two freely independent standard \fpois{} variables with rates $p$ and $(1 - p)$, respectively.
\end{proposition}

\section{Characterization results}\label{sec:characterization}
\subsection{Classical probability}\label{sec:conv_classical}
We first tackle classical Poisson thinning.
\begin{theorem}\label{thm:conv_classical}
	Suppose that $N$ is a non-negative integer-valued random variable defined on some probability space $(\Omega, \mathcal{F}, \P)$ with $\P(N = 0) < 1$, and, on the same space, $X_1, X_2, \ldots$ is a sequence of i.i.d. random variables with finite mean, independent of $N$. Suppose that $S_N:=\sum_{i=1}^N X_i$ is independent of $N - S_N$. 
	\begin{enumerate}
		\item[(a)] If $X_1 \sim \mathrm{Bernoulli}(p)$ for some $p \in (0, 1)$, then $N \sim \mathrm{Poisson}(\lambda)$ for some $\lambda > 0$.
		\item[(b)] If $N \sim \mathrm{Poisson}(\lambda)$ for some $\lambda > 0$, then $X_1 \sim \mathrm{Bernoulli}(p)$ for some $p \in (0, 1)$.
	\end{enumerate}
\end{theorem}
\begin{proof} 
	(a) Let $G(z)=\E(z^N)$ be the probability generating function (PGF) of $N$. Then $G$ is continuous on $[0, 1]$ and analytic on $(0, 1)$. The joint PGF of $S_N$ and $N - S_N$ is given by 
	\[
	H(z_1, z_2) := \E(z_1^{S_N} z_2^{N - S_N}) = G(p z_1 + (1 - p) z_2),
	\]
	which exists for $z_1, z_2 \in [0, 1]$. Now, since $N$ and $N - S_N$ are independent, we must have that
	\[
	H(z_1, z_2) = H(z_1, 1) H(1, z_2).
	\]
	Therefore $G$ must satisfy the following functional equation:
	\[
	G(p z_1 + (1 - p) z_2) = G(p z_1 + (1 - p)) G(p + (1 - p) z_2),
	\]
	where $z_1, z_2 \in [0, 1]$. Setting $g(z) = \log G(z + 1)$ for $z \in [-1, 0]$, this becomes Cauchy's functional equation:
	\[
	g(s_1 + s_2) = g(s_1) + g(s_2),
	\]
	where $s_1 = pz_1 + (1 - p) - 1, s_2 = p + (1 - p)z_2 - 1$ and $s_1, s_2, s_1 + s_2 \in [-1, 0]$. Now $g$ is differentiable on $(-1, 0)$. Taking derivatives with respect to $s_1$, we get that
	\[
	g'(s_1 + s_2) = g'(s_1)
	\]
	for all $s_1, s_2 \in (-1, 0)$ such that $s_1 + s_2 \in (-1, 0)$. Thus $g'$ is constant on $(-1, 0)$, say $\lambda$. Hence, for some constant $c$,
	\[
	g(s) = \lambda s + c
	\]
	for all $s \in (-1, 0)$. As $g$ is continuous at $0$ from the left and $g(0-) = 0$, $c$ must be $0$. Therefore 
	\[
	G(z)=e^{\lambda(z - 1)}
	\]
	for all $z \in (0, 1)$. By continuity at $0$ from the right, 
	\[
	e^{-\lambda} = G(0+) = \P(N = 0) < 1,
	\]
	which implies that $\lambda > 0$. We have thus shown that $G(z) = e^{\lambda(z - 1)}$ on $[0, 1]$ for some $\lambda > 0$, which is equivalent to saying that $N$ follows a Poisson distribution with mean $\lambda.$
	
	(b) Let $\psi(t) = \E e^{\iota t X_1}, t \in \R,$ be the characteristic function (CF) of $X_1$, where $\iota$ denotes the imaginary unit. The joint CF of $S_N$ and $N - S_N$ is given by
	\begin{align*}
	\eta(t_1, t_2) &:= \E e^{\iota (t_1 S_N + t_2 (N - S_N))} \\
	&= \E(e^{\iota t_2 N} \E( e^{\iota (t_1 - t_2) S_N} \mid N)) \\
	&= \E e^{\iota t_2 N} \psi(t_1 - t_2)^N \\
	&= e^{\lambda(e^{\iota t_2}  \psi(t_1 - t_2) - 1)}.
	\end{align*}
	As $S_N$ and $N-S_N$ are independent, we must have
	\[
	\eta(t_1, t_2) = \eta(t_1, 0) \eta(0, t_2)
	\]
	for all $(t_1, t_2)$. We can rewrite this as
	\[
	e^{\lambda(e^{\iota t_2} \psi(t_1 - t_2) + 1 - \psi(t_1) - e^{\iota t_2} \psi(-t_2))} = 1.
	\]
	This means that 
	\[
	I(t_1, t_2) := e^{\iota t_2} \psi(t_1 - t_2) + 1 - \psi(t_1) - e^{\iota t_2} \psi(-t_2) \in \bigg\{ \frac{2\pi n}{\lambda} \mid n \in \Z\bigg\}.
	\]
	Since $I$ defines a continuous function, its range is connected and so $I$ must be a constant. As $I(0, 0) = 0$, we conclude that $I(t_1, t_2) = 0$ for all $t_1, t_2 \in \R$. The assumption that $X_1$ has a finite mean implies that $\psi(t)$ is continuously differentiable at every $t$. Therefore $I$ is differentiable with respect to $t_1$ for any fixed $t_2$. Hence, for all $t_2$,
	\begin{equation}\label{eq:diff_eq_bern}
	0 = \frac{\partial I}{\partial t_1} (0, t_2) = e^{\iota t_2} \psi'(- t_2) - \psi'(0).
	\end{equation}
	Now letting $p = \E X_1$, we have $\psi'(0) = \iota p$. We rewrite \eqref{eq:diff_eq_bern} as
	\[
	\psi'(t) = \iota p e^{\iota t}, t \in \R.
	\]
	Hence, for some constant $d$,
	\[
	\psi(t) = p e^{\iota t} + d.
	\]
	As $\psi(0) = 1$, we must have $d = 1 - p$ so that
	\begin{equation}\label{eq:charfn_bern}
	\psi(t) = p e^{\iota t} + (1 - p).
	\end{equation}
	We now argue that $p \in [0, 1]$. Since $\psi$ as given in \eqref{eq:diff_eq_bern} is infinitely differentiable, all moments of $X_1$ exist. In particular,
	\[
	p = \frac{\psi''(0)}{\iota^2} = \E X_1^2 \ge 0.
	\]
	Now the CF of $1 - X_1$ equals
	\[
	e^{it} \psi(-t) = p + (1 - p) e^{\iota t}.
	\]
	Therefore, arguing as before,
	\[
	1 - p = \E(1 - X_1)^2 \ge 0.
	\]
	All in all, we have shown that the CF of $X_1$ is given by
	\[
	\psi(t) = p e^{\iota t} + (1 - p)
	\]
	for some $p \in [0, 1]$. It follows that $X_1 \sim \mathrm{Bernoulli}(p)$.
\end{proof}

Just as the independence of $S_N$ and $N - S_N$ ties the Poisson and Bernoulli distributions together, so does the ``Poisson-ness'' of $S_N$.
\begin{proposition}\label{prop:conv_classical}
	Suppose that $N$ is a non-negative integer-valued random variable defined on some probability space $(\Omega, \mathcal{F}, \P)$ with $\P(N = 0) < 1$, and, on the same space, $X_1, X_2, \ldots$ is a sequence of i.i.d. random variables, independent of $N$. Suppose that $S_N:=\sum_{i=1}^N X_i$ is Poisson distributed. 
	\begin{enumerate}
		\item[(a)] If $X_1 \sim \mathrm{Bernoulli}(p)$ for some $p \in (0, 1)$, then $N \sim \mathrm{Poisson}(\lambda)$ for some $\lambda > 0$.
		\item[(b)] If $N \sim \mathrm{Poisson}(\lambda)$ for some $\lambda > 0$, then $X_1 \sim \mathrm{Bernoulli}(p)$ for some $p \in (0, 1)$.
	\end{enumerate}
	\begin{proof}
		(a) Because of the Wald identity $\E S_N = \E N \E X_1$, $\E N < \infty$. Let $\E N = \lambda$, and $G(z) = \E z^N$ be the PGF of $N$. Then
		\[
		G(p z + (1 - p)) = \E z^{S_N} = e^{\lambda p (z - 1)}.
		\]
		Letting $u = p z + (1 - p) \in [1 - p, 1]$, we have
		\[
		G(u) = e^{\lambda(u - 1)}
		\]
		for all $u \in [1 - p, 1]$. As $p \in (0, 1)$, the analyticity of $G(z)$ implies that its coefficients must match those of $e^{\lambda(z - 1)}$, which means that $N \sim \mathrm{Poisson}(\lambda)$.
		
		(b) Let $\E S_N = \mu$ and $\psi(t) = \E e^{\iota tX_1}$ be the CF of $X_1$. Now 
		\[
		\E e^{\iota tS_N} = \E(\E(e^{\iota tS_N} \mid N)) = \E \psi(t)^N = e^{\lambda(\psi(t) - 1)}.
		\]
		Thus we have
		\[
		e^{\lambda(\psi(t) - 1)} = \E e^{\iota tS_N} = e^{\mu(e^{\iota t} - 1)},
		\]
		i.e. $e^{J(t)} = 1$ where $J(t) = \lambda(\psi(t) - 1) - \mu (e^{\iota t} - 1)$. As $J(t)$ is continuous, and $J(0) = 0$, the same argument as in the proof of part (b) of Theorem~\ref{thm:conv_classical} implies that $J(t) = 0$ for all $t \in \R$. This means that $\psi(t) = 1 - p + p e^{\iota t}$, where $p = \frac{\lambda}{\mu} > 0$. That $p \le 1$ follows as before by noting that the CF of $(1 - X_1)$ must be $p + (1 - p) e^{\iota t}$.
	\end{proof}
\end{proposition}

\subsection{Free probability}\label{sec:conv_free}
We now turn to \fpois{} thinning. We first prove a free analogue of Craig's theorem.
\begin{proposition}[Craig's theorem, free version]\label{prop:craig_free}
    Suppose that $\p$ is \fpois{} with rate $1$ and is freely independent of $\{\a_1, \ldots, \a_k\}$. Then $\p\a_1, \ldots, \p\a_k$ are freely independent if and only if
    \[
        \state(\a_{i_1}\ldots\a_{i_m}) = 0
    \]
    for any $i_1, \ldots, i_m \in \{1, \ldots, k\}$, $m \ge 2$, such that $i_{\ell}$'s are not all equal.

    In particular, if $\a_{j}$'s are all self-adjoint, and $\state$ is faithful, then $\a_r\a_s = \zero$ for all $r \ne s$.
\end{proposition}
\begin{proof}
    By \eqref{eq:cumulant_two}, we have
    \begin{align*}
        \kappa_m(\p\a_{i_1}, \ldots, \p \a_{i_m}) &= \sum_{\pi \in NC(m)} \kappa_\pi[\p, \ldots, \p] \kappa_{K(\pi)}[\a_{i_1}, \ldots, \a_{i_m}] \\
                                                  &= \state(\p)^m \sum_{\pi \in NC(m)} \kappa_{K(\pi)} [\a_{i_1}, \ldots, \a_{i_m}] \\
                                                  &= \state(\p)^m \state(\a_{i_1}\ldots\a_{i_m}).
    \end{align*}
    Therefore mixed free cumulants of $\p\a_1, \ldots, \p\a_k$ vanish if and only if the stated requirement holds.

    If $\a_j$'s are self-adjoint, we have, for all $r \ne s$,
    \[
        \state((\a_r\a_s)^*(\a_r \a_s)) = \state(\a_s \a_r^2 \a_s) = 0.
    \]
    As $\state$ is faithful, we must have $\a_r\a_s = \zero$.
\end{proof}
\begin{theorem}\label{thm:conv_free}
	Suppose $\p$ and $\b$ are two freely independent elements in a non-commutative probability space $(\cA, \state)$ such that $\p\b$ and $\p(\unit - \b)$ are free.
	\begin{enumerate}
		\item[(a)] If $\b$ is a \fbern{} variable with parameter $p \in (0, 1)$ and $\state(\p) > 0$, then $\p$ must be \fpois{} with rate $1$ and jump size $\state(\p)$.
		\item[(b)] If $\p$ is a \fpois{} variable with rate $1$, then $\state(\b^n) = \state(\b)$ for all $n \ge 1$.
		If, further, the moments of $\b$ match the moments of some probability measure $\nu$ on $\R$, then $\b$ must be a \fbern{} variable with parameter $p = \nu(\{1\})$.  Moreover, if $(\cA, \state)$ is a $*$-probability space with a faithful state $\state$ and $\b$ is self-adjoint, then $\b$ must be a projection.
	\end{enumerate}
\end{theorem}
\begin{proof}
	(a) We will use induction to show that $\kappa_n(\p) = (\kappa_1(\p))^n = (\state(\p))^n$ for all $n \ge 1$. Normalizing $\p$ by $\state(\p)$ if necessary, we may assume that $\kappa_1(\p) = \state(\p) = 1$. Then we need to prove that $\kappa_n(\p) = 1$ for all $n \ge 1$. This is trivially true for $n = 1$. Assuming that it holds for all $n \le N$, consider the case $n = N + 1$. We have, using \eqref{eq:cumulant_two}, 
	\begin{equation}\label{eq:cum_aba}
	\kappa_{N + 1}(\p\b) = \sum_{\pi \in NC(N + 1)} \kappa_{\pi}[\p, \ldots, \p] \kappa_{K(\pi)}[\b, \ldots, \b].
	\end{equation}
	If $\pi \ne \unit$, the blocks of $\pi$ will all have size $\le N$. This, together with the induction hypothesis, implies that $\kappa_{\pi}[\p, \ldots, \p] = 1$ for all $\pi \ne \unit$. Therefore
	\begin{align*}
	\text{the RHS of \eqref{eq:cum_aba}} & = p^{N + 1} \kappa_{N + 1}(\p) + \sum_{\substack{\pi \in NC(N + 1) \\ \pi \ne \bone}} \kappa_{K(\pi)}[\b, \ldots, \b] \\
	& = - p^{N + 1} (1 - \kappa_{N + 1}(\p)) + \sum_{\pi \in NC(N + 1)} \kappa_{K(\pi)}[\b, \ldots, \b] \\
	& = - p^{N + 1} (1 - \kappa_{N + 1}(\p)) + \state(\b^{N + 1}) \\
	& = p - p^{N + 1} (1 - \kappa_{N + 1}(\p)).
	\end{align*}
	Similarly,
	\[
	\kappa_{N + 1}(\p(\unit - \b)) = (1 - p) - (1 -p)^{N + 1} (1 - \kappa_{N + 1}(\p)).
	\]
	Since $\p\b$ and $\p(\unit - \b)$ are free, we must have
	\begin{align*}
	\kappa_{N + 1}(\p) &= \kappa_{N + 1}(\p\b) + \kappa_{N + 1}(\p(\unit - \b)) \\
	&= 1 - (p^{N + 1} + (1 - p)^{N + 1}) (1 - \kappa_{N + 1}(\p)),
	\end{align*}
	i.e.
	\[
	(1 - p^{N + 1} - (1 - p)^{N + 1}) (1 - \kappa_{N + 1}(\p)) = 0.
	\]
	Since $p \in (0, 1)$, we have $ p^{N + 1} + (1 - p)^{N + 1} < p + (1 - p) = 1$. Thus $\kappa_{N + 1}(\p) = 1$. This completes the induction and the proof of (a).
	
    (b) The freeness of $\p\b$ and $\p(\unit - \b)$ implies, by Proposition~\ref{prop:craig_free}, that
    \begin{equation}\label{eq:free_bern_mixed_moments}
        \state(\b^{m_1}(\unit - \b)^{m_2}) = 0
    \end{equation}
    for all $m_1, m_2 \ge 1$. In particular,
    \[
        \state(\b^m(\unit - \b)) = 0
    \]
    for all $m \ge 1$. This means that $\state(\b^m) = \state(\b)$ for all $m \ge 1$. 

    If $\state(\b^m) = \int x^m \, d\nu(x)$ for all $m \ge 1$, then taking $m_1 = m_2 = 2$ in \eqref{eq:free_bern_mixed_moments} gives
    \[
        \int x^2 (1 - x)^2 \, d\nu(x) = 0,
    \]
    which means that $\nu$ must be supported on $\{0, 1\}$. Hence 
    \[
        \state(\b) = \int x \, d\nu(x) = \nu(\{1\}) \in [0, 1],
    \]
    i.e. $\b$ is \fbern{}.

	As for the last claim, since $\b$ is self-adjoint, so is $\a := \b - \b^2$. Now
	\[
	\state(\a^*\a) = \state(\a^2) = \state(\b^2) - 2 \state(\b^3) + \state(\b^4) = 0.
	\]
	As $\state$ is faithful, $\a = \zero$, i.e. $\b^2 = \b$.
\end{proof}

Let $Q_n(z) = 1 - z^n - (1 - z)^n$. Note that the proof of part (a) of Theorem~\ref{thm:conv_free} goes through as long as $Q_n(p) \ne 0$ for all $n \ge 2$. We record this observation in the proposition below. Define the set 
\[
\mathcal{Z} = \{z \in \C \mid Q_n(z) = 0 \text{ for some } n \ge 2\}.
\]
Note that $\mathcal{Z}$ is symmetric about the real axis and contains $0, 1$.
\begin{proposition}\label{prop:gen}
	Suppose $\p$ and $\b$ are two free elements in a non-commutative probability space $(\cA, \state)$. Suppose that $\p\b$ and $\p(\unit - \b)$ are free. If $\b$ is such that $\state(\b^n) = \state(\b) \notin \mathcal{Z}$ for all $n \ge 1$, then $\kappa_n(\p) = \state(\p)^n$ for all $n \ge 1$. In particular, if $\state(\p) > 0$, then $\p$ is \fpois{} with rate $1$ and jump size $\state(\p)$.
\end{proposition}
\begin{figure}[!t]
	\centering
	\includegraphics[scale = 0.30]{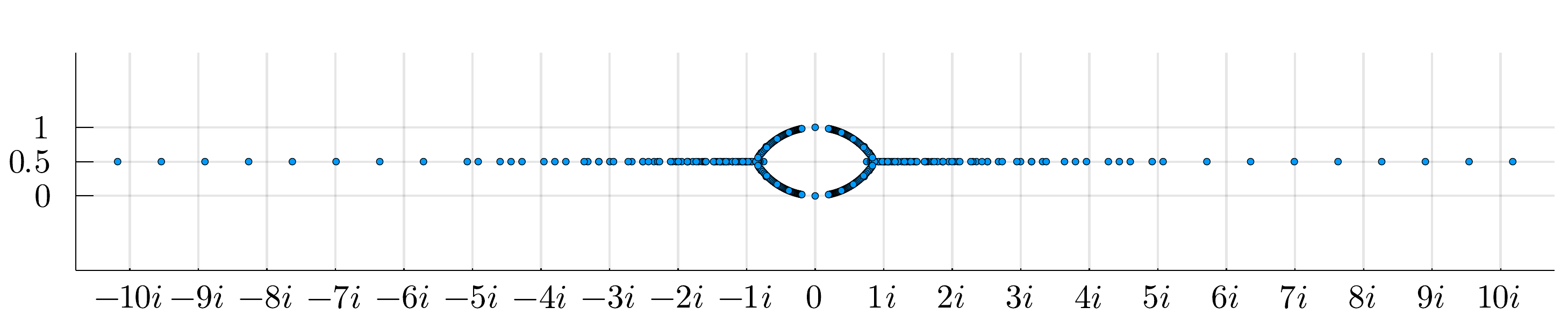}
	\caption{All zeros of $Q_n(z) = 0$ for $2 \le n \le 32$.}
	\label{fig:roots}
\end{figure}
In Figure~\ref{fig:roots}, we plot all the zeros of $Q_n(z) = 0$ for $2 \le n \le 32$. Note that there does not seem to be any real zero other than $0$ and $1$.  Also, Figure~\ref{fig:roots} suggests that $\mathcal{Z}$ intersects the line $\frac{1}{2} + \iota t, t \in \R$, infinitely often. We prove these facts in the following lemma.
\begin{lemma}\label{lem:root}
	For $n \ge 2$, the polynomial $Q_n(z) = 1 - z^n - (1 - z)^n$ has no other real roots than the trivial ones, namely $0$ and $1$. If, further, $n \equiv 0, 1\,(\mathrm{mod}\, 4)$, then $Q_n(z)$ has at least two zeros of the form $\frac{1}{2} + \iota t, t \in \R$.
\end{lemma}
\begin{proof}
	As already noted in the proof of Theorem~\ref{thm:conv_free}, part (a), if $z \in (0, 1)$, then $Q_n(z) > 0$. By symmetry, it suffices to show that $Q_n(z) < 0$ for $z > 1$. This is obvious for $n$ even. So suppose that $n = 2m + 1$, where $m \ge 1$. Then
	\[
	Q_{2m + 1}(z) = (1 - z) (1 + z + \cdots + z^{2m} - (1 - z)^{2m}).
	\]
	As $z > 1$, we have $z^{2m} > (z - 1)^{2m}$. Hence $Q_{2m + 1}(z) < 0$.
	
	As for the last statement, consider
	\begin{align*}
	R_n(t) = Q_n\bigg(\frac{1}{2} + \iota t\bigg) &= 1 - \sum_{0 \le j \le n} \binom{n}{j}\frac{1}{2^j} (\iota^j + (-\iota)^j) t^{n - j} \\
	&= 1 - \sum_{\substack{0 \le j \le n \\ j \equiv 0\,(\mathrm{mod}\, 2)}} 2\binom{n}{j}\frac{1}{2^j} (-1)^{j/2} t^{n - j}.
	\end{align*}
	$R_n(t)$ is a polynomial of even degree with real coefficients and $R_n(0) = Q_n(1/2) > 0$. The desired conclusion follows because, when $n \equiv 0, 1\,(\mathrm{mod}\, 4)$, the leading coefficient of $R_n(t)$ is negative so that $\lim_{t \rightarrow \pm\infty} R_n(t) = -\infty$.
\end{proof}
Now we present an analogue of Proposition~\ref{prop:conv_classical}.
\begin{proposition}\label{prop:conv_free}
	Suppose $\p$ and $\b$ are two freely independent elements in a non-commutative probability space $(\cA, \state)$. Suppose that $\p\b$ is \fpois{} with rate $p \in (0, 1]$ and jump size $\alpha > 0$. 
	\begin{enumerate}
		\item[(a)] If $\b$ is a \fbern{} variable with parameter $p$, then $\p$ must be \fpois{} with rate $1$ and jump size $\alpha$.
		\item[(b)] If $\p$ is a \fpois{} variable with rate $1$ and jump size $\alpha$, then $\b$ must be \fbern{} with parameter $p$.
	\end{enumerate}
\end{proposition}
\begin{proof}
	(a) We have $\kappa_n(\p\b) = p \alpha^n$ for all $n \ge 1$. By freeness, $p \alpha = \kappa_1(\p\b) = \state(\b) \state(\p) = p \state(\p)$, i.e. $\state(\p) = \alpha$. We will show that $\kappa_n(\p) = \alpha^n$ for all $n \ge 1$. We have already seen that this is true when $n = 1$. Assuming its truth for $n \le N$, consider the case $n = N + 1$. We have
	\begin{align*}
	p \alpha^{N + 1} = \kappa_{N + 1}(\p\b) &= \sum_{\pi \in NC(N + 1)} \kappa_{\pi}[\p, \ldots, \p] \kappa_{K(\pi)}[\b, \ldots, \b] \\
	&= \kappa_{N + 1}(\p) p^{N + 1} + \sum_{\substack{\pi \in NC(N + 1) \\ \pi \ne \bone}} \alpha^{N + 1} \kappa_{K(\pi)}[\b, \ldots, \b] \\
	&= (\kappa_{N + 1}(\p) - \alpha^{N + 1}) p^{N + 1} + \alpha^{N + 1} \state(\b^{N + 1}) \\
	&= (\kappa_{N + 1}(\p) - \alpha^{N + 1}) p^{N + 1} + \alpha^{N + 1} p.
	\end{align*}
	As $p \ne 0$, we get that $\kappa_{N + 1}(\p) = \alpha^{N + 1}$. This completes the induction and the proof.
	
    (b) As  $\p$ has rate $1$, we have $\kappa_n(\p\b) = \alpha^n \state(\b^n)$ for all $n \ge 1$ (see the proof of Proposition~\ref{prop:thinning_free}, part (a)). Therefore $\state(\b^n) = p$ for all $n \ge 1$.
\end{proof}

\subsection{Implications for the asymptotic version of Cochran's theorem}
Together Theorem~\ref{thm:conv_free} and Proposition~\ref{prop:conv_free} imply the following result which may be thought of as a converse of Proposition~\ref{prop:asymptotic_cochran_forward_gen}.

\begin{proposition}\label{prop:asymptotic_cochran_backward_gen}
	Suppose $(\bX_n)_{n \ge 1}$ is a sequence of $n \times m$ random matrices whose entries satisfy Assumption (A1), $\frac{m}{n} \rightarrow 1$, and $(\bB_n)_{n \ge 1}$ is a sequence of symmetric deterministic matrices satisfying Assumption (A2) such that $\E\frac{1}{n}\tr \bB_n^k$ converges for all $k \ge 1$. 
	\begin{enumerate}
        \item[(a)] $\frac{1}{n} \bX_n^\top \bB_n \bX_n$ and $\frac{1}{n} \bX_n^\top (\bI_n - \bB_n) \bX_n$ are asymptotically free only if there is some constant $p \in [0, 1]$ such that
            \[
                \E \frac{1}{n}\tr (\bB_n^k) \rightarrow p \text{ for all } k \ge 1.
            \]
		\vskip5pt
		\item[(b)] $\frac{1}{n}\bX_n^\top\bB_n\bX_n$ converges to a \fpois{} variable with rate $p \in (0, 1]$ only if 
            \[
                \E \frac{1}{n}\tr (\bB_n^k) \rightarrow p \text{ for all } k \ge 1.
            \]
	\end{enumerate}
\end{proposition}

Together Propositions~\ref{prop:asymptotic_cochran_forward_gen} and \ref{prop:asymptotic_cochran_backward_gen} constitute the asymptotic analogues of Proposition~\ref{prop:cochran}, Corollary~\ref{cor:cochran} and Proposition~\ref{prop:craig}.

\section{Thinning into more than two components}\label{sec:more_than_two}
We have considered classical and \fpois{} thinning with Bernoulli variables. Instead of a Bernoulli, one can also use categorical variables with more than two categories (i.e. multinomial random variables with a single trial). The classical characterization results extend easily to this case.
\begin{theorem}\label{thm:conv_classical_more_than_two}
	Suppose that $N$ is a non-negative integer-valued random variable defined on some probability space $(\Omega, \mathcal{F}, \P)$ with $\P(N = 0) < 1$, and, on the same space, $X_1, X_2, \ldots$, is a sequence of i.i.d. $k$-dimensional random vectors, $X_i = (X_{i, 1}, \ldots, X_{i, k})^\top$, with finite mean, independent of $N$, such that $\sum_{\ell = 1}^k X_{i, \ell} = 1$. Let $S_N^{(\ell)} := \sum_{i = 1}^N X_{i, \ell}$. Suppose that $S_N^{(1)}, \ldots, S_N^{(k)}$ are independent. 
	\begin{enumerate}
		\item[(a)] If $X_1 \sim \mathrm{Categorical}(p_1, \ldots, p_k)$ such that $p_\ell \in (0, 1)$ for some $\ell$, then $N \sim \mathrm{Poisson}(\lambda)$ for some $\lambda > 0$.
		\item[(b)] If $N \sim \mathrm{Poisson}(\lambda)$ for some $\lambda > 0$, then $X_1 \sim \mathrm{Categorical}(p_1, \ldots, p_k)$ for some $p_1, \ldots, p_k \in [0, 1]$ adding up to $1$.
	\end{enumerate}
\end{theorem}
\begin{proof}
    The proof is a straightforward consequence of Theorem~\ref{thm:conv_classical} and the fact that $S_N^{(\ell)}$ is independent of $N - S_N^{(\ell)} = \sum_{ \ell' \ne \ell} S_N^{(\ell')}$.
\end{proof}
The free analogue of such a categorical variable is a tuple of elements $(\b_1, \ldots, \b_k)$ such that 
\begin{enumerate}
	\item[(i)] $\sum_{\ell = 1}^k \b_{\ell} = \unit$, 
	\item[(ii)] $\state(\b_{\ell}^n) = \state(\b_{\ell}) \in [0, 1]$ for any $1 \le \ell \le k$, $n \ge 1$, and 
	\item[(iii)] $\state(\b_{i_1}^{j_1}\cdots\b_{i_m}^{j_m}) = 0$ for any $i_\ell \in \{1, \ldots, k\}$, $j_\ell \ge 1$, $1 \le \ell \le m$, $m \ge 2$ such that $i_1 \ne i_2 \ne \cdots \ne i_m$.
\end{enumerate}
Letting $p_\ell = \state(\b_\ell)$, we must have $\sum_{\ell = 1}^k p_\ell = 1$. We call such a tuple of elements a \textit{free categorical} tuple with parameter $(p_1, \ldots, p_k)$. 

A canonical example of such a tuple is $k$ mutually orthogonal projections in a $*$-probability space adding up to the identity.

We have the following extension of Proposition~\ref{prop:thinning_free}. The proof is similar and hence omitted.
\begin{proposition}
	Let $\p$ be \fpois{} with rate $\lambda$ and $(\b_1, \ldots, \b_k)$ be a free categorical tuple with parameter $(p_1, \ldots, p_k) \in (0, 1)^k$. Also suppose that $\p$ and $\{\b_1, \ldots, \b_k\}$ are freely independent.
	\begin{enumerate}
        \item[(a)] If $\lambda = 1$, then $\p\b_1, \ldots, \p\b_k$ are free with $\p\b_\ell$ being a \fpois{} with rate $p_\ell$.
        \item[(b)] If $\lambda \ne 1$, then neither are $\p\b_\ell$ \fpois{} nor are they freely independent.
	\end{enumerate}
\end{proposition}
In the converse direction, we have the following result which follows from Proposition~\ref{prop:craig_free} and Theorem~\ref{thm:conv_free} on noting that if $\p\b_1, \ldots, \p\b_k$ are free and $\sum_{\ell = 1}^k \b_\ell = \unit$, then $\p\b_\ell$ and $\p(\unit - \b_\ell)$ are free for any $1 \le \ell \le k$.
\begin{theorem}
	Suppose that $\p$ and $\{\b_1, \ldots, \b_k\}$ are freely independent. Assume that $\sum_{\ell = 1}^k \b_\ell = \unit$. Assume that for each $1 \le \ell \le k$, the moments of $\b_\ell$ match the moments of some probability measure $\nu_\ell$ on $\R$. Suppose that $\p\b_1, \ldots, \p\b_k$ are free.
	\begin{enumerate}
		\item[(a)] If $(\b_1, \ldots, \b_k)$ is free categorical with parameter $(p_1, \ldots, p_k)$ such that $p_i \in (0, 1)$ for some $1 \le i \le k$ and $\state(\p) > 0$, then $\p$ must be \fpois{} with rate $1$ and jump size $\state(\p)$.
		\item[(b)] If $\p$ is a \fpois{} variable with rate $1$, then $(\b_1, \ldots, \b_k)$ must be a free categorical tuple. Moreover, if $(\cA, \state)$ is a $*$-probability space with a faithful state $\state$ and $\b_\ell$ is self-adjoint for any $1 \le \ell \le k$, then the $\b_\ell$'s must be a family of $k$ mutually orthogonal projections.
	\end{enumerate}
\end{theorem}
Clearly, using the results of this section, we can formulate versions of Propositions~\ref{prop:asymptotic_cochran_forward_gen} and \ref{prop:asymptotic_cochran_backward_gen} for more than two quadratic functions of random matrices.

\section{Other variants of \fpois{} thinning}\label{sec:other_versions}
\subsection{An asymptotic variant of \fpois{} thinning}
Recall that a classical Poisson random variable arises as the limit of a sum of independent sparse Bernoulli variables: If $(X_i)_{i = 1}^n$ are independent Bernoulli random variables with parameter $\frac{\lambda}{n}, \lambda > 0$, then the binomial variable $N_n = \sum_{i = 1}^n X_i$ converges in distribution to a Poisson$(\lambda)$ variable as $n \rightarrow \infty$. In fact, one can construct the classical thinning procedure from this perspective as follows. Let $\{X, Y\}$ be independent, where $X$ is a Bernoulli random variable with mean $\frac{\lambda}{n}$. Also, let $\{X_1, Y_1\}, \ldots, \{X_n, Y_n\}$ be independent copies of $\{X, Y\}$. Define
\[
    S_n = \sum_{i = 1}^n X_iY_i.
\]
Then the following result can be proved using arguments similar to the proofs of Theorem~\ref{thm:conv_classical} and Proposition~\ref{prop:conv_classical}.
\begin{theorem}
The following are equivalent.
\begin{enumerate}\label{thm:classical_poi_v2}
    \item[(i)] $S_n$ is asymptotically Poisson.
    \item[(ii)] $Y$ is a Bernoulli random variable.
    \item[(iii)] $S_n$ and $N_n - S_n$ are asymptotically independent.
\end{enumerate}
\end{theorem}

Similar to the classical case, a \fpois{} random variable arises as the limit of a sum of freely independent sparse Bernoulli variables. To elaborate, let $(\a_i)_{i = 1}^n$ be freely independent Bernoulli variables with parameter $\frac{\lambda}{n}$. Then as $n \rightarrow \infty$, the sum $\N_n = \sum_{i = 1}^n \a_i$ converges in distribution to a \fpois{} random variable with rate $\lambda$ (and jump size $1$). Motivated by the construction described in the previous paragraph, we can describe an asymptotic variant of \fpois{} thinning as follows. Let $\{\a, \b\}$ be classically independent Bernoulli variables with parameters $\frac{\lambda}{n}$ and $p$, respectively. Also, let $\{\a_1, \b_1\}, \ldots, \{\a_n, \b_n\}$ be free copies of $\{\a, \b\}$. Define
\[
    \S_n = \sum_{i = 1}^n \a_i\b_i.
\]
This, as we will soon see, converges as $n \rightarrow \infty$ to a \fpois{} variable with rate $\lambda p$. Moreover, $\S_n$ and $\N_n - \S_n$ are asymptotically free. We will also prove a characterization result about this procedure. First we compute an expression for the mixed free cumulants of $\S_n$ and $\N_n - S_n$ for an arbitrary $\b$. 
\begin{proposition}\label{prop:thinning_free_v2_mixed_cum}
    Let $\{\a, \b\}$ be classically independent, where $\a$ is a Bernoulli variable with parameter $\frac{\lambda}{n}$. Also, let $\{\a_1, \b_1\}, \ldots, \{\a_n, \b_n\}$ be free copies of $\{\a, \b\}$. Let $\S_n^{(\epsilon)} := \sum_{i = 1}^n \a_i \b_i^{(\epsilon)}$, where $\epsilon \in \{0, 1\}$ and $\b_{i}^{(1)} = \b_i$, $\b_i^{(0)} = \unit - \b_i$. In this notation, $\S_n^{(1)} = \S_n$ and $\S_n^{(0)} = \N_n - \S_n$. Then for any $m \ge 1$ and $\epsilon_1, \ldots, \epsilon_m \in \{0, 1\}$,
    \begin{equation}\label{eq:thinning_free_v2_mixed_cum}
        \kappa_m(\S_n^{(\epsilon_1)}, \ldots, \S_n^{(\epsilon_m)}) = \lambda \state(\b^{\sum_{i = 1}^m \epsilon_i}(\unit - \b)^{m - \sum_{i = 1}^m \epsilon_i}) + O\bigg(\frac{1}{n}\bigg).
    \end{equation}
\end{proposition}

\begin{proof}
Note that
\begin{align*}
    \kappa_m(\S_n^{(\epsilon_1)}, \ldots, \S_n^{(\epsilon_m)}) &= \kappa_m\bigg(\sum_{i = 1}^n \a_i \b_i^{(\epsilon_1)}, \ldots, \sum_{i = 1}^n \a_i \b_i^{(\epsilon_m)}\bigg) \\
                                                             &= \sum_{i_1, \ldots, i_m} \kappa_m(\a_{i_1}\b_{i_1}^{(\epsilon_1)}, \ldots, \a_{i_m}\b_{i_m}^{(\epsilon_m)}) \\
                                                             &= \sum_{i = 1}^n \kappa_m(\a_i\b_i^{(\epsilon_1)}, \ldots, \a_i\b_i^{(\epsilon_m)}) \\
                                                             &= n \kappa_m(\a\b^{(\epsilon_1)}, \ldots, \a\b^{(\epsilon_m)}),
\end{align*}
where the penultimate equality follows from freeness since all mixed free cumulants must vanish, and the last equality follows from the fact that $\{\a_i, \b_i\}$ are identically distributed as $\{\a, \b\}$. Now
\begin{align*}
    \kappa_m(\a\b^{(\epsilon_1)}, \ldots, \a\b^{(\epsilon_m)}) &= \sum_{\pi \in NC(m)} \moeb(\pi, \bone_m) \state_\pi(\a\b^{(\epsilon_1)}, \ldots, \a\b^{(\epsilon_m)}) \\ 
                                                                       &= \sum_{\pi \in NC(m)} \moeb(\pi, \bone_m) \prod_{V \in \pi} \state(\a^{|V|}) \state(\prod_{\ell \in V}\b^{(\epsilon_{\ell})}) \\
                                                                       &= \sum_{\pi \in NC(m)} \moeb(\pi, \bone_m) \bigg(\frac{\lambda}{n}\bigg)^{|\pi|} \prod_{V \in \pi} \state(\prod_{\ell \in V}\b^{(\epsilon_{\ell})}) \\
                                                                       &= \frac{\lambda}{n} \state(\b^{\sum_{i = 1}^m \epsilon_i}(\unit - \b)^{m - \sum_{i = 1}^m \epsilon_i}) + O\bigg(\frac{1}{n^2}\bigg),
\end{align*}
whence the desired result follows.
\end{proof}
We can now prove the following characterization result.
\begin{theorem}\label{thm:free_poi_v2}
    Let $\{\a, \b\}$ be classically independent, where $\a$ is a Bernoulli variable with parameter $\frac{\lambda}{n}$. Also, let $\{\a_1, \b_1\}, \ldots, \{\a_n, \b_n\}$ be free copies of $\{\a, \b\}$. Define $\N_n$ and $\S_n$ as before. Then we have the following.
\begin{enumerate}
    \item[(i)] $\S_n$ is asymptotically \fpois{} with rate $\mu > 0$ and jump size $\alpha > 0$ if and only if $\state((\alpha^{-1}\b)^m) = \frac{\mu}{\lambda}$ for all $m \ge 1$.
    \item[(ii)] $\S_n$ and $\N_n - \S_n$ are asymptotically free if and only if $\state(\b^m) = \state(\b)$ for all $m \ge 1$.
\end{enumerate}
\end{theorem}
\begin{proof}
    (i) From \eqref{eq:thinning_free_v2_mixed_cum} we see that $\kappa_m(\S_n) = \kappa_m(\S_n^{(1)}, \ldots, \S_n^{(1)}) = \lambda\state(\b^m) + O(n^{-1})$ for any $m \ge 1$. Thus $\S_n$ is asymptotically \fpois{} with rate $\lambda$ and jump size $\alpha$ if and only if for all $m \ge 1$,
\[
    \lambda\state(\b^m) = \mu \alpha^m,
\]
i.e. $\state((\alpha^{-1}\b)^m) = \frac{\mu}{\lambda}$.

(ii) Note that by \eqref{eq:thinning_free_v2_mixed_cum}, all mixed free cumulants of $\S_n$ and $\N_n - \S_n$ vanish asymptotically if and only if $\state(\b^{m_1}(\unit - \b)^{m_2}) = 0$ for any $m_1, m_2 > 0$. This is equivalent to the requirement that $\state(\b^m) = \state(\b)$ for all $m \ge 1$.
\end{proof}
\begin{remark}
Note that the variant of \fpois{} thinning considered in this subsection does not require $\lambda$ to be $1$. 
\end{remark}
\subsection{A \fpois{} random measure approach}
One of the anonymous referees proposed a version of \fpois{} thinning based on the notion of \fpois{} random measures \cite{barndorff2006classical}. In this subsection, we will briefly describe the suggested construction without formally defining \fpois{} random measures. The starting point is the representation of a Poisson random variable as
\[
    N = \int_{[0, \lambda]} m(dt),
\]
where $m$ is a Poisson random measure on $([0, \infty), \mathcal{B}_{[0, \infty)}, \leb)$ and $\leb$ denotes the Lebesgue measure. Let $X_i \overset{\iid}{\sim} \nu$. Then we can represent $N$ as
\[
    N = \int_{[0, \lambda] \times \R} M(dt, dx),
\]
where $M$ is a Poisson random measure on $([0, \infty) \times \R, \mathcal{B}_{[0, \infty) \times \R}, \leb \otimes \nu)$. With this we can represent the sum $S_N$ as
\[
    S_N = \int_{[0, \lambda] \times \R} x M(dt, dx).
\]
By Theorem~\ref{thm:conv_classical} and Proposition~\ref{prop:conv_classical}, we then have the equivalence of the following statements.
\begin{enumerate}
    \item[(i)] $S_N$ is Poisson.
    \item[(ii)] $\nu$ is a Bernoulli distribution.
    \item[(iii)] $S_N$ and $N - S_N$ are independent.
\end{enumerate}

Barndorff-Nielsen and Thorbj{\o}rnsen \cite{barndorff2006classical} defined \fpois{} random measures on measure spaces and a theory of integration with respect to them. With the notion of \fpois{} random measures, one can define an exact analogue of Poisson thinning in the free world. Let $\M$ be a \fpois{} random measure on $([0, \infty) \times \R, \mathcal{B}_{[0, \infty) \times \R}, \leb \otimes \nu)$. Then a \fpois{} variable with rate $\lambda$ may be represented as
\[
    \p = \int_{[0, \lambda] \times \R} \M(dt, dx).
\]
Define
\[
    \S_\p := \int_{[0, \lambda] \times \R} x\, \M(dt, dx),
\]
which can be thought of as the analogue of $S_N$ in the free world. It then seems plausible that the following statements are equivalent.
\begin{enumerate}
    \item[(i)] $\S_\p$ is a \fpois{} variable.
    \item[(ii)] $\nu$ is a Bernoulli distribution.
    \item[(iii)] $\S_\p$ and $\p - \S_\p$ are freely independent.
\end{enumerate}

\section*{Acknowledgments}
The author thanks the anonymous referees for their helpful comments and suggestions which significantly improved the paper.

\bibliographystyle{abbrv}
\bibliography{poi-bern-arxiv}
\end{document}